\documentclass[letterpaper, 10pt, twocolumn]{ieeeconf}      

\IEEEoverridecommandlockouts                              
\def\BibTeX{{\rm B\kern-.05em{\sc i\kern-.025em b}\kern-.08em
    T\kern-.1667em\lower.7ex\hbox{E}\kern-.125emX}}
    
\usepackage[left=54pt, right=54pt,bottom=54pt, top=54pt]{geometry}

\usepackage{amsmath,mathrsfs,amsfonts,amssymb,graphicx,epsfig,stmaryrd}
 \usepackage{amsthm}
\usepackage{subcaption}
\usepackage{caption}
\usepackage{color,multirow,rotating}
\usepackage{algorithm,algpseudocode,algorithmicx}
\usepackage{cite,url,framed,bm,balance,dsfont,varwidth}
\usepackage{nicefrac,stmaryrd}

\newtheorem{theorem}{Theorem}

\newtheorem{remark}{Remark}

\usepackage{cleveref}
\usepackage{tikz}
\usetikzlibrary{calc,trees,positioning,arrows,chains,shapes.geometric,decorations.pathreplacing,decorations.pathmorphing,shapes, matrix,shapes.symbols}

\newcommand{\differential}{{\rm{d}}}
\renewcommand{\det}{{\mathrm{det}}}

\newcommand{\RNum}[1]{\uppercase\expandafter{\romannumeral #1\relax}}

\allowdisplaybreaks

\title{\LARGE\textbf{
Optimal Mass Transport over the Euler Equation}
}

\author{Charlie Yan, Iman Nodozi, Abhishek Halder
\thanks{Charlie Yan and Iman Nodozi are with the Department of Electrical and Computer Engineering, University of California, Santa Cruz, CA 95064, USA,
        {\tt\small{\{cyan140,inodozi\}@ucsc.edu}}.\\
        Abhishek Halder is with the Department of Applied Mathematics, University of California, Santa Cruz, CA 95064, USA,
        {\tt\small{ahalder@ucsc.edu}}.\\
        This work was partially supported by NSF grant 2112755.
}}

\begin{document}

\maketitle
\pagenumbering{gobble}

\bstctlcite{IEEEexample:BSTcontrol} 

\begin{abstract}
We consider the finite horizon optimal steering of the joint
state probability distribution subject to the angular velocity
dynamics governed by the Euler equation. The problem and its solution
amounts to controlling the spin of a rigid body via feedback, and is
of practical importance, for example, in angular stabilization of a
spacecraft with stochastic initial and terminal states. We clarify how
this problem is an instance of the optimal mass transport (OMT)
problem with bilinear prior drift. We deduce both static and dynamic
versions of the Eulerian OMT, and provide analytical and numerical
results for the synthesis of the optimal controller.
\end{abstract}


\section{Introduction}\label{sec:introduction}
The controlled angular velocity dynamics for a rotating rigid body such as a spacecraft, is given by the well-known Euler equation
\begin{align}
\bm{J}\dot{\bm{\omega}} = -[\bm{\omega}]^{\times}\bm{J\omega} + \bm{\tau},
\label{EulerEqn}
\end{align}
where the positive diagonal matrix $\bm{J}:={\rm{diag}}(J_1,J_2,J_3)$ comprises of the principal moments of inertia, the vector $\bm{\omega}:=(\omega_1,\omega_2,\omega_3)^{\top}\in\mathbb{R}^{3}$ denotes the body's angular velocity (in rad/s) along its principal axes, the vector $\bm{\tau}:=(\tau_1,\tau_2,\tau_3)^{\top}\in\mathbb{R}^{3}$ denotes the torque input applied about the principal axes, and $$[\bm{\omega}]^{\times}:=\begin{pmatrix}
0 & -\omega_3 & \omega_2\\
\omega_3 & 0 & -\omega_1\\
-\omega_2 & \omega_1 & 0
\end{pmatrix}
\in\mathfrak{so}(3).$$
As usual, $\mathfrak{so}(3)$ denotes the Lie algebra of the three dimensional rotation group ${\rm{SO}}(3)$. Motivated by the problem of steering the probabilistic uncertainties in angular velocities over a prescribed time horizon, we consider the deterministic and stochastic variants of the optimal mass transport (OMT) \cite{benamou2000computational,villani2003topics,villani2009optimal} over the Euler equation, which we refer to as the OMT-EE. 

Specifically, let $\mathcal{P}_2(\mathbb{R}^3)$ denote the manifold of probability measures supported on $\mathbb{R}^3$ with finite second moments. Given two probability measures $\mu_0,\mu_T\in\mathcal{P}_2(\mathbb{R}^3)$, the \emph{deterministic} OMT-EE associated with \eqref{EulerEqn} is a stochastic optimal control problem:
\begin{subequations}
\begin{align}
&\underset{\bm{u}\in\mathcal{U}}{\inf}\int_{0}^{T}\mathbb{E}_{\mu^{\bm{u}}}\left[q(\bm{x}^{\bm{u}}) + r(\bm{u})\right]\differential t \label{dOMTobj}\\
&\text{subject to}\quad\dot{\bm{x}}^{\bm{u}} = \bm{\alpha}\odot\bm{f}(\bm{x}^{\bm{u}}) + \bm{\beta}\odot\bm{u}, \; i\in\llbracket 3\rrbracket:=\{1,2,3\}, \label{dOMTdyn}\\
&\mu^{\bm{u}}(\bm{x}^{\bm{u}},t=0)=\mu_0\:\text{(given)}, \: \mu^{\bm{u}}(\bm{x}^{\bm{u}},t=T)=\mu_T\:\text{(given)}, \label{InitialAndTerminalPDF}
\end{align}
\label{detOMTEE} 
\end{subequations}
where the fixed time horizon is $[0,T]$ for some prescribed $T>0$, and $\mathbb{E}_{\mu^{\bm{u}}}\left[\cdot\right]$ denotes the expectation w.r.t. the controlled state probability measure $\mu^{\bm{u}}(\bm{x}^{\bm{u}},t)$ for $t\in[0,T]$, i.e., $\mathbb{E}_{\mu^{\bm{u}}}\left[\cdot\right] := \int (\cdot)\:\differential\mu^{\bm{u}}$. The superscript $\bm{u}$ for a variable indicates that variable's dependence on the choice of control $\bm{u}$. The symbol $\odot$ denotes the elementwise (Hadamard) vector product. 

The correspondence between \eqref{EulerEqn} and \eqref{dOMTdyn} follows by noting that the controlled state $\bm{x}^{\bm{u}}\equiv\:\text{controlled}\;\bm{\omega}$, the control $\bm{u}\equiv\bm{\tau}$, the vector field 
\begin{align}
\bm{f}(\bm{z}):=(z_2 z_3, z_3 z_1, z_1 z_2)^{\top}\,\text{for}\;\bm{z}\in\mathbb{R}^3,
\label{DriftVectorField}    
\end{align}
and the parameter vectors $\bm{\alpha},\bm{\beta}\in\mathbb{R}^3$ have entries 
\begin{align}
    \alpha_i := (J_{i+1\mod 3} - J_{i+2\mod 3})/J_i, \; \beta_i := 1/J_i, \; i\in\llbracket 3\rrbracket.
    \label{parameterVector}
\end{align}
The cost-to-go in \eqref{dOMTobj} comprises of an additive state cost $q(\cdot)$, and a strictly convex and superlinear (i.e., 1-coercive) control cost $r(\cdot)$. Of particular interest is the case $q(\cdot)\equiv 0$ and $r(\cdot)=\frac{1}{2}\|\cdot\|_2^2$ which corresponds to minimum effort control. We suppose that $q+r$ is lower bounded.

Let $\Omega$ be the space of continuous functions $\eta:[0,T]\mapsto\mathbb{R}^3$, which is a complete separable metric space endowed with the topology of uniform convergence on compact time
intervals. With $\Omega$, we associate the $\sigma$-algebra $\mathscr{F}=\sigma\{\eta(s)\mid 0\leq s \leq T\}$, and consider the complete filtered probability space $\left(\Omega,\mathscr{F},\mathbb{P}\right)$ with filtration $\mathscr{F}_t=\sigma\{\eta(s)\mid 0\leq s \leq t\leq T\}$. So, $\mathscr{F}_0$ contains all $\mathbb{P}$-null sets and $\mathscr{F}_t$ is right continuous. The stochastic initial condition $\bm{x}^{u}(t=0)$ in \eqref{detOMTEE} is $\mathscr{F}_0$ measurable. For a given control policy $\bm{u}$, the controlled state $\bm{x}^{u}(t)$ is $\mathscr{F}_t$-adapted (i.e., non-anticipating) for all $t\in[0,T]$.

In \eqref{detOMTEE}, the set of feasible Markovian control policies 
\begin{align}
\mathcal{U}:=\{\bm{u} : \mathbb{R}^3 \times [0,T] \mapsto \mathbb{R}^{3} &\mid \!\!\int_{0}^{T}\!\!\mathbb{E}_{\mu^{\bm{u}}}\left[r(\bm{u})\right] \differential t < \infty\}.
\label{FeasibleControl}    
\end{align}
Thus, solving \eqref{detOMTEE} amounts to designing an admissible Markovian control policy $\bm{u}\in\mathcal{U}$ that transfers the stochastic angular velocity state from a prescribed initial to a prescribed terminal probability measure under the controlled sample path dynamics constraint \eqref{dOMTdyn}, and hard deadline constraint. The initial and terminal measures can be interpreted as the estimated and allowable statistical uncertainty specifications, respectively, and therefore, problem \eqref{detOMTEE} asks to directly control or reshape uncertainties in a nonparametric sense \cite{chen2021controlling}. The objective of this paper is to study problem \eqref{detOMTEE}, and its \emph{stochastic} version where \eqref{dOMTdyn} may have additive process noise (discussed in Sec. \ref{sec:stocOMTEE}).


\subsubsection*{Contributions}\label{Seccontrib}
This work makes the following specific contributions.
\begin{itemize}
    
    \item We clarify the connections and differences of the OMT over the Euler equation vis-\`{a}-vis the classical OMT, from both static and dynamic perspectives.
    
    \item We present a numerical method to solve the minimum energy steering problem via neural networks with Sinkhorn losses for the probability density function (PDF)-level nonparametric boundary conditions.
\end{itemize}

\subsubsection*{Organization}\label{subsec:organization}
This paper is organized as follows. After reviewing the OMT preliminaries in Sec. \ref{Sec:OMT}, we study the deterministic OMT-EE in Sec. \ref{sec:detOMTEE}, i.e., when the Euler equation has no additive process noise. Sec. \ref{sec:uncontrolledPDF} discusses the unforced PDF evolution subject to the (deterministic) Euler equation. In Sec. \ref{sec:stocOMTEE}, we focus on the minimum energy OMT-EE subject to the Euler equation with additive process noise, and in Sec. \ref{secPINN} that follows, we detail a neural network framework to solve the corresponding necessary conditions of optimality. Sec. \ref{sec:numerical} provides numerical simulation results for a case study. Sec. \ref{sec:conclusions} concludes the paper.


\subsubsection*{Related Works}\label{subsec:relatedworks}
Continuous time deterministic optimal control subject to the angular velocity dynamics given by the Euler equation, has been studied in several prior works. In the finite horizon setting, Athans et. al. \cite{athans1963time} derived the minimum time, minimum fuel (assuming free terminal time) and minimum energy (assuming fixed terminal time) controllers--all steering an arbitrary initial angular velocity vector to zero. Considering free terminal state and no terminal cost, Kumar \cite{kumar1965optimum} showed that a tangent hyperbolic feedback is optimal for finite horizon problem w.r.t. quadratic state and quadratic control cost-to-go. Again considering free terminal state, Dwyer \cite{dwyer1982control} derived the optimal finite horizon controller w.r.t. quadratic state and quadratic control cost-to-go, as well as quadratic terminal cost. Infinite horizon optimal control problem w.r.t. quadratic state and control objective was studied in \cite{tsiotras1998optimal}.

Besides control design, systems-theoretic properties for \eqref{EulerEqn} are known too. Thanks to the periodicity of unforced motion, \eqref{EulerEqn} enjoys global controllability guarantees and it is known \cite[Thm. 4 and disussions thereafter]{brockett1976nonlinear}, \cite[Reamrk in p. 895]{baillieul1980geometry} that the controlled dynamics is reachable on entire $\mathbb{R}^{3}$. See also \cite{crouch1984spacecraft}.

Formulating and solving the OMT with prior dynamics is a relatively recent endeavor, see e.g., \cite{chen2016optimal,elamvazhuthi2018optimal,caluya2020finite,ito2022sinkhorn}. To the best of the authors' knowledge, OMT over the Euler equation has not been investigated before.


\subsubsection*{Notations}\label{SecNotations} We use boldfaced small letters for vectors and boldfaced capital letters for matrices. When a probability measure $\mu$ is absolutely continuous, it admits a PDF $\rho$, and $\differential\mu(\cdot)=\rho\differential(\cdot)$. We use $\sharp$ to denote the pushforward of a probability measure or PDF (when the measure is absolutely continuous). The symbol $\sim$ is used as a shorthand for ``follows the probability distribution". We use $\langle\cdot,\cdot\rangle$, $\nabla$ and $\Delta$ to denote the standard Euclidean inner product, the Euclidean gradient and the Euclidean Laplacian, respectively. In case of potential confusion, we put a subscript to $\nabla$ to clarify w.r.t. which variable the gradient is being taken; otherwise we omit the subscript. We use $\mathcal{N}\left(\bm{m},\bm{\Sigma}\right)$ to denote a joint normal PDF with mean vector $\bm{m}$ and covariance matrix $\bm{\Sigma}$. The symbol $\bm{I}_3$ denotes the $3\times 3$ identity matrix.

\section{OMT Preliminaries}\label{Sec:OMT}
To ease the ensuing development, we now summarize rudiments on classical OMT. Well-known references for this topic are \cite{villani2003topics,villani2009optimal}; for a brief summary see e.g., \cite{halder2014geodesic}.

The \emph{static formulation} of OMT goes back to Gaspard Monge in 1781, which concerns with finding a mass preserving transport map $\bm{\theta}:\mathbb{R}^{n}\mapsto\mathbb{R}^{n}$ pushing a given measure $\mu_0$ to another $\mu_T$ while minimizing a transportation cost $\int_{\mathbb{R}^{n}}c(\bm{x},\bm{\theta}(\bm{x}))\differential\mu_0$ where $c$ is some ground cost functional. A common choice for $c$ is half of the squared Euclidean distance, but in general, the choice of the functional $c$ plays an important role for guaranteeing existence-uniqueness of the minimizer $\bm{\theta}^{\rm{opt}}(\cdot)$.

Even when the existence-uniqueness of the \emph{optimal transport map} $\bm{\theta}^{\rm{opt}}(\cdot)$ can be guaranteed, Monge's formulation requires solving a nonlinear nonconvex problem over all measurable pushforward mappings $\bm{\theta}:\mathbb{R}^{n}\mapsto\mathbb{R}^{n}$ taking $\mu_0$ to $\mu_T$. For $c(\bm{x},\bm{y})\equiv \frac{1}{2}\|\bm{x}-\bm{y}\|_2^2$, $\bm{x},\bm{y}\in\mathbb{R}^{n}$, $\differential\mu_0(\bm{x})=\rho_0(\bm{x})\differential\bm{x}, \differential\mu_T(\bm{x})=\rho_T(\bm{y})\differential\bm{y}$, it is known \cite{brenier1991polar} that $\bm{\theta}^{\rm{opt}}$ exists, is unique, and admits a representation $\bm{\theta}^{\rm{opt}} = \nabla\psi$ for some convex function $\psi$. Even then, the direct computation of $\psi$ is numerically challenging because it reduces to solving a second order nonlinear elliptic Monge-Amp\`{e}re PDE \cite[p. 126]{villani2003topics}: $\det\left(\nabla^2\psi(\bm{x})\right)\rho_{T}(\nabla\psi(\bm{x}))=\rho_0(\bm{x})$, where $\det$ and $\nabla^2$ denote the determinant and the Hessian, respectively.

A more tractable reformulation of the static OMT is due to Leonid Kantorovich in 1942 \cite{kantorovich1942translocation}, which instead of finding the optimal transport map $\bm{\theta}^{\rm{opt}}$, seeks to compute an \emph{optimal coupling} $\pi^{\rm{opt}}$ between the given measures $\mu_0,\mu_T$ that solves
\begin{align}
\underset{\pi\in\Pi_2(\mu_0,\mu_T)}{\arg\inf}\int_{\mathbb{R}^{n}\times\mathbb{R}^{n}}c(\bm{x},\bm{y})\differential\pi(\bm{x},\bm{y})
\label{KantorovichLP}
\end{align}
where $\Pi_2(\mu_0,\mu_T)$ denotes the set of all joint probability measures $\pi$ supported over the product space $\mathbb{R}^{n}\times\mathbb{R}^{n}$ with $\bm{x}$ marginal $\mu_0$, and $\bm{y}$ marginal $\mu_T$. Notice that \eqref{KantorovichLP} is an infinite dimensional linear program. The map $\bm{\theta}^{\rm{opt}}$ is precisely the support of the optimal coupling $\pi^{\rm{opt}}$. In the other direction, we can recover $\pi^{\rm{opt}}$ from $\bm{\theta}^{\rm{opt}}$ as $\pi^{\rm{opt}} = \left({\rm{Id}}\times\bm{\theta}^{\rm{opt}}\right)\:\sharp\:\mu_0$ where ${\rm{Id}}$ denotes the identity map.

The \emph{dynamic formulation} of OMT due to Benamou and Brenier \cite{benamou2000computational} appeared at the turn of the 21st century. When $c(\bm{x},\bm{y})\equiv \frac{1}{2}\|\bm{x}-\bm{y}\|_2^2$ and $\mu_0,\mu_T$ admit respective PDFs $\rho_0,\rho_T$, the dynamic formulation is the following stochastic optimal control problem:
\begin{subequations}
\begin{align}
&\underset{(\rho^{\bm{u}},\bm{u})\in\mathcal{P}_2(\mathbb{R}^n)\times\mathcal{U}}{\arg\inf}\int_{0}^{T}\int_{\mathbb{R}^{n}}\frac{1}{2}\|\bm{u}\|_2^2\:\rho^{\bm{u}}(\bm{x}^{\bm{u}},t)\differential\bm{x}^{\bm{u}}\differential t\label{BBobj}\\
&\qquad\dfrac{\partial\rho^{\bm{u}}}{\partial t} + \nabla_{\bm{x}^{\bm{u}}}\cdot\left(\rho^{\bm{u}}\bm{u}\right) = 0,\label{BBdyn}\\
&\qquad\rho^{\bm{u}}(\bm{x}^{\bm{u}},t=0) = \rho_0, \quad \rho^{\bm{u}}(\bm{x}^{\bm{u}},t=T) = \rho_T.\label{BBendpoints}
\end{align}
\label{BenamouBrenierOMT}  
\end{subequations}
The constraint \eqref{BBdyn} is the \emph{Liouville PDE} (see e.g., \cite{halder2011dispersion}) that governs the evolution of the state PDF $\rho^{\bm{u}}(\bm{x}^{\bm{u}},t)$ under a feasible control policy $\bm{u}\in\mathcal{U}$. So \eqref{BenamouBrenierOMT} is a problem of optimally steering a given joint PDF $\rho_0$ to another $\rho_T$ over time horizon $[0,T]$ using a vector of single integrators, i.e., with full control authority in $\mathcal{U}$. The solution $(\rho^{\rm{opt}},\bm{u}^{\rm{opt}})$ for \eqref{BenamouBrenierOMT} satisfies
\begin{subequations}
\begin{align}
&\rho^{\rm{opt}}(\bm{x}^{\bm{u}},t) = \bm{\theta}_t \:\sharp\: \rho_0, \quad \bm{\theta}_t := \!\left(\!1-\frac{t}{T}\!\right)\!{\rm{Id}} + \frac{t}{T}\bm{\theta}^{\rm{opt}}, \label{rhoopt}\\
&\bm{u}^{\rm{opt}}(\bm{x}^{\bm{u}},t) = \nabla_{\bm{x}^{\bm{u}}}\phi(\bm{x}^{\bm{u}},t), \quad \dfrac{\partial\phi}{\partial t} + \frac{1}{2}\|\nabla_{\bm{x}^{\bm{u}}}\phi\|_{2}^{2} = 0. \label{uopt}
\end{align}
\label{BBsolution}
\end{subequations}
Thus, \eqref{rhoopt} tells that the optimally controlled PDF is obtained as pushforward of the initial PDF via a map that is a linear interpolation between identity and the optimal transport map. Consequently, the PDF $\rho^{\rm{opt}}$ itself is a (nonlinear) McCann's displacement interpolant \cite{mccann1997convexity} between $\rho_0$ and $\rho_T$. The optimal control in \eqref{uopt} is obtained as the gradient of the solution of a \emph{Hamilton-Jacobi-Bellman (HJB) PDE}.

The $\psi(\bm{x})$ in static OMT and the $\phi(\bm{x},t)$ in dynamic OMT are related \cite[Thm. 5.51]{villani2003topics} through the \emph{Hopf-Lax representation formula}
\begin{subequations}
\begin{align}
\phi(\bm{x},t) &= \!\underset{\bm{y}\in\mathbb{R}^n}{\inf}\!\left(\phi(\bm{y},0) + \dfrac{1}{2t}\|\bm{x}-\bm{y}\|_2^2\right), \: t\in(0,T],\\
\phi(\bm{y},0) &= \psi(\bm{y}) -  \frac{1}{2}\|\bm{y}\|_2^2,
\end{align}
\label{HopfLax}
\end{subequations}
i.e., $\phi(\bm{x},t)$ is the Moreau-Yosida proximal envelope \cite[Ch. 3.1]{parikh2014proximal} of $\phi(\bm{y},0) = \psi(\bm{y}) - \frac{1}{2}\|\bm{y}\|_2^2$, and hence $\phi(\bm{x},t)$ is continuously differentiable w.r.t. $\bm{x}\in\mathbb{R}^{n}$.

Classical OMT allows defining a distance metric, called the \emph{Wasserstein metric} $W$, on the manifold of probability measures or PDFs. In particular, when $c(\bm{x},\bm{y})\equiv \frac{1}{2}\|\bm{x}-\bm{y}\|_2^2$, the infimum value achieved in \eqref{KantorovichLP} is the one half of the squared Wasserstein metric between $\mu_0$ and $\mu_T$, i.e.,
\begin{align}
W^{2}\!\left(\mu_0,\mu_T\right)\! := \!\!\underset{\pi\in\Pi_2(\mu_0,\mu_T)}{\inf}\!\!\int_{\mathbb{R}^{n}\times\mathbb{R}^{n}}\!\!\!\!\|\bm{x}-\bm{y}\|_2^2\: \differential\pi(\bm{x},\bm{y}),
\label{DefWassContinuous}    
\end{align}
which is also equal to the infimum value achieved in \eqref{BenamouBrenierOMT}, provided $\mu_0,\mu_T$ are absolutely continuous. The tuple $\left(\mathcal{P}_2\left(\mathbb{R}^{n}\right),W\right)$ defines a complete separable metric space, i.e., a polish space. This offers a natural way to metrize the topology of weak convergence of probability measures w.r.t. the metric $W$.

For a regularization parameter $\varepsilon>0$, we refer to the entropy-regularized version of \eqref{DefWassContinuous} as \emph{Sinkhorn divergence}
\begin{align}
W_{\varepsilon}^{2}\left(\mu_0,\mu_T\right) := &\underset{\pi\in\Pi_2(\mu_0,\mu_T)}{\inf}\int_{\mathbb{R}^{n}\times\mathbb{R}^{n}}\big\{\|\bm{x}-\bm{y}\|_2^2 \nonumber\\
&\qquad\qquad+\varepsilon\log\pi(\bm{x},\bm{y})\big\}\differential\pi(\bm{x},\bm{y}).
\label{DefSinkDiv}    
\end{align}
As $\varepsilon\downarrow 0$, the Sinkhorn divergence \eqref{DefSinkDiv} approaches the Wasserstein metric \eqref{DefWassContinuous}. 

\section{Deterministic OMT-EE}\label{sec:detOMTEE}
We suppose that the endpoint measures $\mu_0,\mu_T\in\mathcal{P}_2(\mathbb{R}^3)$ in \eqref{detOMTEE} are absolutely continuous with respective PDFs $\rho_0,\rho_T$, and rewrite \eqref{detOMTEE} as
\begin{subequations}
\begin{align}
&\underset{(\rho^{\bm{u}},\bm{u})\in\mathcal{P}_2(\mathbb{R}^3)\times\mathcal{U}}{\arg\inf}\int_{0}^{T}\!\!\int_{\mathbb{R}^{3}}\!\!\left(q(\bm{x}^{\bm{u}})+r(\bm{u})\right)\:\rho^{\bm{u}}(\bm{x}^{\bm{u}},t)\differential\bm{x}^{\bm{u}}\differential t\label{detOMTEErhoobj}\\
&\qquad\dfrac{\partial\rho^{\bm{u}}}{\partial t} + \nabla_{\bm{x}^{\bm{u}}}\cdot\left(\rho^{\bm{u}}\left(\bm{\alpha}\odot\bm{f}(\bm{x}^{\bm{u}}) + \bm{\beta}\odot\bm{u}\right)\right) = 0,\label{detOMTEErhodyn}\\
&\qquad\rho^{\bm{u}}(\bm{x}^{\bm{u}},t=0) = \rho_0, \quad \rho^{\bm{u}}(\bm{x}^{\bm{u}},t=T) = \rho_T.\label{detOMTEErhoendpoints}
\end{align}
\label{detOMTEErho}  
\end{subequations}
Problem \eqref{detOMTEErho} generalizes \eqref{BenamouBrenierOMT} in two ways. \emph{First}, unlike \eqref{BBdyn}, the constraint \eqref{detOMTEErhodyn} has a prior nonlinear drift $\bm{f}$ given by \eqref{DriftVectorField}. \emph{Second}, the cost \eqref{detOMTEErhoobj} is more general than \eqref{BBobj}. We refer to \eqref{detOMTEErho} as the \emph{dynamic OMT-EE}. 

We clarify here that the solution of the Liouville PDE \eqref{detOMTEErhodyn} is understood in the weak sense, i.e., for all compactly supported smooth test functions $\zeta(\bm{x}^{\bm{u}},t)\in C_{c}^{
\infty}\left([0,T]\times\mathbb{R}^3\right)$, the function $\rho^{\bm{u}}(\bm{x}^{\bm{u}},t)$ satisfies $\int_{0}^{T}\int_{\mathbb{R}^{3}}\left(\rho^{\bm{u}}\frac{\partial\zeta}{\partial t} + \rho^{\bm{u}}\langle\bm{\alpha}\odot\bm{f}(\bm{x}^{\bm{u}}) + \bm{\beta}\odot\bm{u},\nabla_{\bm{x}^{\bm{u}}}\zeta\rangle\right)\differential\bm{x}^{\bm{u}}\differential t + \int_{\mathbb{R}^{3}}\rho_0(\bm{x}^{\bm{u}})\zeta(\bm{x}^{\bm{u}},t=0)\differential\bm{x}^{\bm{u}} = 0$. 

\subsection{Static OMT-EE}
At this point, a natural question arises: if \eqref{detOMTEErho} is the Euler equation generalization of the dynamic OMT \eqref{BenamouBrenierOMT}, then what is the corresponding generalization of the static OMT \eqref{KantorovichLP}? 

To answer this, we slightly generalize the setting: we replace $\mathbb{R}^{n}$ in \eqref{KantorovichLP} with an $n$ dimensional Riemannian manifold $\mathcal{M}$. Consider an absolutely continuous curve $\bm{\gamma}(t)\in\mathcal{M}$, $t\in[0,T]$, and $(\bm{\gamma},\dot{\bm{\gamma}})\in\mathcal{TM}$ (tangent bundle). Then for $\bm{x},\bm{y}\in\mathcal{M}$, we think of $c(\bm{x},\bm{y})$ in \eqref{KantorovichLP} to be derived from a Lagrangian $L:[0,T]\times\mathcal{TM}\mapsto\mathbb{R}$, i.e., express $c$ as an action integral
\begin{align}
c(\bm{x},\bm{y}) = \underset{\bm{\gamma}(\cdot)\in\Gamma_{\bm{xy}}}{\inf}\displaystyle\int_{0}^{T}L(t,\bm{\gamma}(t),\dot{\bm{\gamma}}(t))\:\differential t,   
\label{ActionIntegral}    
\end{align}
where 
\begin{align*}
\Gamma_{\bm{xy}}:=\{\bm{\gamma}:[0,T]\mapsto \mathbb{R}^{n} \mid &\, \bm{\gamma}(\cdot)\;\text{is absolutely continuous},\nonumber\\
&\qquad \bm{\gamma}(0)=\bm{x},\bm{\gamma}(1)=\bm{y}\}.
\end{align*}
In particular, the choice $\mathcal{M}\equiv\mathbb{R}^{n}$ and $L(t,\bm{\gamma},\dot{\bm{\gamma}})\equiv\frac{1}{2}\|\dot{\bm{\gamma}}\|_2^2$ results in $c(\bm{x},\bm{y}) = \|\bm{x}-\bm{y}\|_2^2$, i.e., the standard Euclidean OMT.

For OMT-EE, $\mathcal{M}\equiv\mathbb{R}^{3}$ and we have the Lagrangian 
\begin{align}
L(t,\bm{\gamma},\dot{\bm{\gamma}})\equiv q(\bm{\gamma})+r(\left(\dot{\bm{\gamma}} - \bm{\alpha}\odot\bm{f}\right)\oslash\bm{\beta})   
\label{LagrangianOMTEE}    
\end{align}
where $\oslash$ denotes vector element-wise (Hadamard) division. In particular, $L$ in \eqref{LagrangianOMTEE} has no explicit dependence on $t$, i.e., $L:\mathcal{TM}\mapsto\mathbb{R}$.

This identification allows us to define the \emph{static OMT-EE} as the linear program
\begin{align}
\underset{\pi\in\Pi_2(\mu_0,\mu_T)}{\arg\inf}\int_{\mathbb{R}^{3}\times\mathbb{R}^{3}}c(\bm{x},\bm{y})\differential\pi(\bm{x},\bm{y})
\label{KantorovichLPEE}
\end{align}
where $c$ is given by \eqref{ActionIntegral}-\eqref{LagrangianOMTEE}, and $\Pi_2(\mu_0,\mu_T)$ denotes the set of all joint probability measures $\pi$ supported over the product space $\mathbb{R}^{3}\times\mathbb{R}^{3}$ with $\bm{x}$ marginal $\mu_0$, and $\bm{y}$ marginal $\mu_1$. We next show that identifying \eqref{LagrangianOMTEE} also helps establish the existence-uniqueness of solution for \eqref{detOMTEErho}. 

\subsection{Back to Dynamic OMT-EE}
We have the following result for problem \eqref{detOMTEErho}.
\begin{theorem}\label{thm:existenceuniqueness}
(\textbf{Existence-uniqueness}) Let $r:\mathbb{R}^{3}\mapsto\mathbb{R}_{\geq 0}$ be strictly convex and superlinear function. Then the minimizing tuple $(\rho^{\rm{opt}},\bm{u}^{\rm{opt}})$ for problem \eqref{detOMTEErho} exists and is unique.
\end{theorem}
\begin{proof}
Since $r$ is strictly convex, so $L$ in \eqref{LagrangianOMTEE} viewed as function of $\dot{\bm{\gamma}}\in\mathbb{R}^{3}$, is strictly convex composed with an affine map. Therefore, $L$ is strictly convex in $\dot{\bm{\gamma}}$.

We next show that $L$ in \eqref{LagrangianOMTEE} is also superlinear in $\dot{\bm{\gamma}}\in\mathbb{R}^{3}$. To see this, notice that
\begin{align}
\lim_{\|\dot{\bm{\gamma}}\|_{2}\rightarrow\infty} \dfrac{L}{\|\dot{\bm{\gamma}}\|_{2}} &= \lim_{\|\dot{\bm{\gamma}}\|_{2}\rightarrow\infty} \dfrac{r(\left(\dot{\bm{\gamma}} - \bm{\alpha}\odot\bm{f}\right)\oslash\bm{\beta})}{\|\dot{\bm{\gamma}}\|_{2}}\nonumber\nonumber\\
&= \lim_{\|\bm{z}\|_{2}\rightarrow\infty}\dfrac{r(\bm{z})}{\|\bm{\alpha}\odot\bm{f} + \bm{\beta}\odot\bm{z}\|_2}.
\label{LimitToCompute}
\end{align}
Using triangle inequality: $\|\bm{\alpha}\odot\bm{f} + \bm{\beta}\odot\bm{z}\|_2 \leq \|\bm{\alpha}\odot\bm{f}\|_2 + \|\bm{\beta}\odot\bm{z}\|_2 \leq \|\bm{\alpha}\odot\bm{f}\|_2 + \|\bm{\beta}\|_\infty\|\bm{z}\|_2$, and hence
$$\dfrac{r(\bm{z})}{\|\bm{\alpha}\odot\bm{f} + \bm{\beta}\odot\bm{z}\|_2} \geq \dfrac{r(\bm{z})}{\|\bm{\alpha}\odot\bm{f}\|_2 + \|\bm{\beta}\|_\infty\|\bm{z}\|_2}.$$
Taking the limit $\|\bm{z}\|_2\rightarrow\infty$ to both sides of above, we obtain
\begin{align*}
\eqref{LimitToCompute} &\geq \lim_{\|\bm{z}\|_{2}\rightarrow\infty}\dfrac{r(\bm{z})}{\|\bm{\alpha}\odot\bm{f}\|_2 + \|\bm{\beta}\|_\infty\|\bm{z}\|_2}\\
&= \lim_{\|\bm{z}\|_{2}\rightarrow\infty}\dfrac{r(\bm{z})/\|\bm{z}\|_2}{ \|\bm{\beta}\|_\infty} = +\infty,
\end{align*}
since $r$ is superlinear, and $\|\bm{\alpha}\odot\bm{f}\|_2,\|\bm{\beta}\|_\infty > 0$. Thus, \eqref{LimitToCompute} itself equals to $+\infty$, thereby proving that $L$ is indeed superlinear.

The Lagrangian \eqref{LagrangianOMTEE} being both strictly convex and superlinear in $\dot{\bm{\gamma}}$, is a weak \emph{Tonelli Lagrangian} \cite[p. 118]{villani2009optimal}, \cite[Ch. 6.2]{figalli2007optimal}, and therefore guarantees \cite[Thm. 1.4.2]{figalli2007optimal} the existence and uniqueness of the minimizing pair $(\rho^{\rm{opt}},\bm{u}^{\rm{opt}})$ for problem \eqref{detOMTEErho}.
\end{proof}
\begin{remark}
The cost $c$ in \eqref{ActionIntegral} being derived from a weak Tonelli Lagrangian \eqref{LagrangianOMTEE}, as shown in the proof above, equivalently guarantees the existence-uniqueness of the solution for the static OMT-EE \eqref{KantorovichLPEE}.
\end{remark}

\subsection{The Case $q(\cdot)\equiv 0$, $r=\frac{1}{2}\|\cdot\|_2^2$}\label{subsec:detomteeminenergy}
A specific instance of \eqref{detOMTEErho} that is of practical interest is \emph{minimum energy angular velocity steering}, i.e., the case 
$$q(\cdot)\equiv 0, \quad r=\frac{1}{2}\|\cdot\|_2^2.$$ 
Then, \eqref{detOMTEErho} resembles the Benamou-Brenier dynamic OMT \eqref{BenamouBrenierOMT} except that the controlled Liouville PDE \eqref{detOMTEErhodyn} has a prior bilinear drift which \eqref{BBdyn} does not have.

\begin{theorem}\label{thm:FirstOrderOptimalityDynOMTEE}
(\textbf{Necessary conditions of optimality for minimum energy steering of angular velocity PDF without process noise}) The optimal tuple $(\rho^{\rm{opt}},\bm{u}^{\rm{opt}})$ solving problem \eqref{detOMTEErho} with $q(\cdot)\equiv 0$, $r=\frac{1}{2}\|\cdot\|_2^2$, satisfies the following first order necessary conditions of optimality:
\begin{subequations}
\begin{align}
&\frac{\partial \phi}{\partial t}+\frac{1}{2}\left\|\bm{\beta}\odot\nabla_{\bm{x}^{\bm{u}}}\phi\right\|_{2}^{2}+\langle\nabla_{\bm{x}^{\bm{u}}}\phi, \bm{\alpha}\odot\bm{f}(\bm{x}^{\bm{u}})\rangle = 0, \label{DetHJBPDE}\\
&\dfrac{\partial\rho^{{\rm{opt}}}}{\partial t} \!+\! \nabla_{\bm{x}^{\bm{u}}}\cdot\left(\rho^{{\rm{opt}}}\!\left(\bm{\alpha}\odot\!\bm{f}(\bm{x}^{\bm{u}}) + \bm{\beta}^{2}\odot\!\nabla_{\bm{x}^{\bm{u}}}\phi\right)\right) \!=\!0,\label{detrho}  \\
&\rho^{{\rm{opt}}}(\bm{x}^{\bm{u}},t=0) = \rho_0, \quad \rho^{{\rm{opt}}}(\bm{x}^{\bm{u}},t=T) = \rho_T,\label{detBoundaryConditions}\\
& \bm{u}^{\rm{opt}} = \bm{\beta}\odot \nabla_{\bm{x}^{\bm{u}}}\phi, \label{detOptimalControl}
\end{align}
\label{detOptimalityConditions}
\end{subequations}
where $\bm{\beta}^{2}$ denotes the vector element-wise square.
\end{theorem}
\begin{proof}
Consider problem \eqref{detOMTEErho} with $q\equiv 0$, $r(\cdot)=\frac{1}{2}\|\cdot\|_2^2$, and its associated Lagrangian
\begin{align}
&\mathcal{L}\left(\rho^{\bm{u}}, \bm{u},\phi\right) := \!\!\int_{0}^{T}\!\!\!\!\int_{\mathbb{R}^{3}}\!\!\bigg\{\frac{1}{2}\|\bm{u}(\bm{x}^{\bm{u}},t)\|_2^2\:\rho^{\bm{u}}(\bm{x}^{\bm{u}},t)+\phi(\bm{x}^{\bm{u}},t)\nonumber\\
&\left(\!\dfrac{\partial\rho^{\bm{u}}}{\partial t} + \nabla_{\bm{x}^{\bm{u}}}\cdot\left(\rho^{\bm{u}}\left(\bm{\alpha}\odot\bm{f}(\bm{x}^{\bm{u}}) + \bm{\beta}\odot\bm{u}\right)\right)\!\right)\!\!\bigg\}\differential\bm{x}^{\bm{u}}\differential t
\label{LagrangianDetProblem}
\end{align}
where the Lagrange multiplier $\phi\in C^{1}\left(\mathbb{R}^{3};[0,T]\right)$. Let $\mathcal{P}_{0T}$ denote the family of PDF-valued curves over $[0,T]$ satisfying \eqref{detBoundaryConditions}. We perform unconstrained minimization of \eqref{LagrangianDetProblem} over $\mathcal{P}_{0T}\times\mathcal{U}$. 

Performing integration-by-parts of the right-hand-side of \eqref{LagrangianDetProblem} and assuming the limits for $\|\bm{x}^{\bm{u}}\|_2\rightarrow\infty$ are zero, we arrive at the unconstrained minimization of
\begin{align}
&\int_{0}^{T}\!\!\!\!\int_{\mathbb{R}^{3}}\!\!\left(\frac{1}{2}\|\bm{u}(\bm{x}^{\bm{u}},t)\|_2^2 -\dfrac{\partial\phi}{\partial t} - \langle\nabla_{\bm{x}^{\bm{u}}}\phi,\bm{\alpha}\odot\bm{f}(\bm{x}^{\bm{u}}) \right.\nonumber\\
&\left.\qquad\qquad\qquad\qquad+ \bm{\beta}\odot\bm{u}\rangle\right)\rho^{\bm{u}}(\bm{x}^{u},t)\:\differential\bm{x}^{\bm{u}}\:\differential t.
\label{LagrangianSimplified}
\end{align}
Pointwise minimization of the integrand in \eqref{LagrangianSimplified} w.r.t. $\bm{u}$ for each fixed PDF-valued curve in $\mathcal{P}_{0T}$, gives
$$\bm{u}^{\rm{opt}} = {\rm{diag}}\left(\bm{\beta}\right) \nabla_{\bm{x}^{\bm{u}}}\phi,$$
which is the same as \eqref{detOptimalControl}. Substituting the above expression for optimal control back in \eqref{LagrangianSimplified}, and equating the resulting expression to zero, we obtain the dynamic programming equation
\begin{align}
&\int_{0}^{T}\!\!\!\!\int_{\mathbb{R}^{3}}\!\!\left(-\dfrac{\partial\phi}{\partial t} - \frac{1}{2}\left\|\bm{\beta}\odot\nabla_{\bm{x}^{\bm{u}}}\phi\right\|_{2}^{2}-\langle\nabla_{\bm{x}^{\bm{u}}}\phi, \bm{\alpha}\odot\bm{f}(\bm{x}^{\bm{u}})\rangle\right)\nonumber\\
&\qquad\qquad\qquad\qquad\qquad\qquad\rho^{\bm{u}}(\bm{x}^{u},t)\:\differential\bm{x}^{\bm{u}}\:\differential t = 0.
\label{DPeqn}
\end{align}
For \eqref{DPeqn} to hold for any feasible $\rho^{\bm{u}}(\bm{x}^{u},t)$, the expression within the parentheses must vanish, which gives us the HJB PDE \eqref{DetHJBPDE}.

Since $\rho^{{\rm{opt}}}$ must satisfy the feasibility conditions \eqref{detOMTEErhodyn}-\eqref{detOMTEErhoendpoints}, hence substituting \eqref{detOptimalControl} therein yields \eqref{detrho}-\eqref{detBoundaryConditions}. 
\end{proof}

\begin{remark}
Equations \eqref{detOptimalControl} and \eqref{DetHJBPDE} generalize the condition \eqref{uopt} in classical dynamic OMT. The solution of the coupled system of HJB PDE \eqref{DetHJBPDE} and Liouville PDE \eqref{detrho} with boundary conditions \eqref{detBoundaryConditions} yields the optimal PDF $\rho^{\rm{opt}}$.
\end{remark}


\section{Uncontrolled PDF Evolution}\label{sec:uncontrolledPDF}Before delving into the approximate numerical solution for the optimally controlled PDF evolution, we briefly remark on the \emph{uncontrolled} PDF evolution. Specifically, we show next that the bilinear structure of the drift vector field in Eulerian dynamics \eqref{dOMTdyn} allows analytic handle on the uncontrolled PDFs, which will come in handy later for comparing the optimally controlled versus uncontrolled evolution of the stochastic states.

In the absence of control, we denote the uncontrolled state vector as $\bm{x}$, and the uncontrolled joint state PDF as $\rho$ (i.e., without the $\bm{u}$ superscripts). In that case, \eqref{detOMTEErhodyn} specializes to the uncontrolled Liouville PDE
\begin{align}
\dfrac{\partial\rho}{\partial t} + \nabla_{\bm{x}}\cdot\left(\rho\bm{\alpha}\odot\bm{f}(\bm{x})\right)=0.
\label{UncontrolledLiouville}
\end{align}
Since the drift in \eqref{dOMTdyn} is divergence free, we can explicitly solve \eqref{UncontrolledLiouville} with known initial condition $\rho(\bm{x},t=0)=\rho_0$, as
\begin{align}
\rho(\bm{x},t) = \rho_{0}\left(\bm{x}_0\left(\bm{x},t\right)\right)
\label{SolnUncontrolledLiouville}
\end{align}
where $\bm{x}_0\left(\bm{x},t\right)$ is the \emph{inverse flow map} associated with the unforced initial value problem: 
\begin{align}
\dot{\bm{x}}=\bm{\alpha}\odot\bm{f}(\bm{x}), \quad \bm{x}(t=0)=\bm{x}_0.
\label{UnforcedEulerODEivp}
\end{align}

For an asymmetric rigid body, we have $J_1 \neq J_2 \neq J_3$, and the corresponding flow map $\bm{x}\left(\bm{x}_0,t\right)$ for \eqref{UnforcedEulerODEivp} is given component-wise by (see e.g., \cite[equation (37.10)]{landau1976course})
\begin{subequations}
\begin{align}
x_1 &= \overline{x}_{10} \:{\rm{cn}}\left(\omega_p t + \lambda_1, \lambda_2\right),\\
x_2 &= \overline{x}_{20} \:{\rm{sn}}\left(\omega_p t + \lambda_1, \lambda_2\right),\\
x_3 &= \overline{x}_{30} \:{\rm{dn}}\left(\omega_p t + \lambda_1, \lambda_2\right),
\end{align}
\label{FlowMapAsymmetric}
\end{subequations}
where ${\rm{cn}}$ (elliptic cosine), ${\rm{sn}}$ (elliptic sine), ${\rm{dn}}$ (delta amplitude) are the Jacobi elliptic functions, and the variables $\overline{x}_{i0} \forall i\in\llbracket 3\rrbracket$, $\omega_p$, $\lambda_1,\lambda_2$ depend only on $\bm{x}_0$. In Sec. \ref{sec:numerical}, we numerically compute the inverse flow map $\bm{x}_0\left(\bm{x},t\right)$ associated with \eqref{FlowMapAsymmetric}.

For an axisymmetric rigid body, we have $J_1 = J_2 \neq J_3$, and the inverse flow map $\bm{x}_0\left(\bm{x},t\right)$ for \eqref{UnforcedEulerODEivp} can be computed component-wise analytically as
\begin{subequations}
\begin{align}
\gamma &:= \dfrac{x_2 - x_{1}\tan\left(\alpha_2x_3 t\right)}{x_1 + x_2 \tan\left(\alpha_2x_3 t\right)},\\
x_{10} &= \left(\dfrac{x_1^2 + x_2^2}{1 + \gamma^2}\right)^{\frac{1}{2}},\\
x_{20} &= \gamma\: x_{10} = \gamma\left(\dfrac{x_1^2 + x_2^2}{1 + \gamma^2}\right)^{\frac{1}{2}},\\
x_{30} &= x_{3},
\end{align}
\label{FlowMapAxisymmetric}
\end{subequations}
and thus \eqref{SolnUncontrolledLiouville} takes the form
$$\rho(x_1,x_2,x_3,t)=\rho_0\left(\left(\dfrac{x_1^2 + x_2^2}{1 + \gamma^2}\right)^{\frac{1}{2}},\gamma\left(\dfrac{x_1^2 + x_2^2}{1 + \gamma^2}\right)^{\frac{1}{2}},x_3\right).$$
We eschew the computation details for brevity.


\section{Minumum Energy Stochastic OMT-EE}\label{sec:stocOMTEE}
To facilitate the numerical solution of the dynamic OMT-EE discussed in Sec. \ref{subsec:detomteeminenergy}, i.e., the solution of \eqref{detOMTEErho} with $q\equiv 0$, $r(\cdot)\equiv\frac{1}{2}\|\cdot\|_2^2$, we perturb the sample path dynamics \eqref{dOMTdyn} with an additive process noise resulting in the It\^{o} stochastic differential equation (SDE):
\begin{equation}
\begin{aligned}
  \differential \bm{x}^{\bm{u}}=\! \left(\bm{\alpha}\odot\bm{f}(\bm{x}^{\bm{u}})  + \bm{\beta}\odot\bm{u}\right) \differential t + \sqrt{2 \delta}\:\differential\bm{w}, \; \delta > 0.
\end{aligned}
\label{SDE1}
\end{equation}
The $\bm{w}$ in \eqref{SDE1} denotes standard Wiener process in $\mathbb{R}^{3}$. Due to process noise, the first order Liouville PDE \eqref{detOMTEErhodyn} is replaced by the second order Fokker-Planck-Kolmogorov (FPK) PDE
\begin{align}
\dfrac{\partial\rho^{\bm{u}}}{\partial t} \!+\! \nabla_{\bm{x}^{\bm{u}}}\cdot\left(\rho^{\bm{u}}\left(\bm{\alpha}\odot\bm{f}(\bm{x}^{\bm{u}}) + \bm{\beta}\odot\bm{u}\right)\right) \!=\! \delta\Delta_{\bm{x}^{\bm{u}}}\rho^{\bm{u}},
\label{FPK}
\end{align}
which has both advection and diffusion. The corresponding necessary conditions of optimality are then transformed as follows.
\begin{theorem}\label{thm:FirstOrderOptimalitySBPEE}
(\textbf{Necessary conditions of optimality for minimum energy steering of angular velocity PDF with process noise}) Let $\delta>0$. The optimal tuple $(\rho^{\rm{opt}},\bm{u}^{\rm{opt}})$ solving problem \eqref{detOMTEErho} with $q(\cdot)\equiv 0$, $r=\frac{1}{2}\|\cdot\|_2^2$, and \eqref{detOMTEErhodyn} replaced by \eqref{FPK}, satisfies the following first order necessary conditions of optimality:
\begin{subequations}
\begin{align}
&\frac{\partial \phi}{\partial t}+\frac{1}{2}\left\|\bm{\beta}\odot\nabla_{\bm{x}^{\bm{u}}}\phi\right\|_{2}^{2}+\langle\nabla_{\bm{x}^{\bm{u}}}\phi, \bm{\alpha}\odot\bm{f}(\bm{x}^{\bm{u}})\rangle\nonumber\\
&\qquad\qquad\qquad\qquad\qquad\qquad\qquad=-\delta\Delta_{\bm{x}^{\bm{u}}}\phi, \label{StocHJBPDE}\\
&\dfrac{\partial\rho^{{\rm{opt}}}}{\partial t} \!+\! \nabla_{\bm{x}^{\bm{u}}}\cdot\left(\rho^{{\rm{opt}}}\!\left(\bm{\alpha}\odot\!\bm{f}(\bm{x}^{\bm{u}}) + \bm{\beta}^{2}\odot\!\nabla_{\bm{x}^{\bm{u}}}\phi\right)\right) \nonumber\\
&\qquad\qquad\qquad\qquad\qquad\qquad\qquad=\delta\Delta_{\bm{x}^{\bm{u}}}\rho^{{\rm{opt}}},\label{Stocrho}\\
&\rho^{{\rm{opt}}}(\bm{x}^{\bm{u}},t=0) = \rho_0, \quad \rho^{{\rm{opt}}}(\bm{x}^{\bm{u}},t=T) = \rho_T,\label{StocBoundaryConditions}\\
& \bm{u}^{\rm{opt}} = \bm{\beta}\odot \nabla_{\bm{x}^{\bm{u}}}\phi, \label{StocOptimalControl}
\end{align}
\label{StocOptimalityConditions}
\end{subequations}
where $\bm{\beta}^{2}$ denotes the vector element-wise square.
\end{theorem}
\begin{proof}
The proof follows the same line of arguments, mutatis mutandis, as in the proof of Theorem \ref{thm:FirstOrderOptimalityDynOMTEE}. See also \cite[proof of Prop. 1]{caluya2021wasserstein}, \cite[p. 275]{chen2021stochastic}.
\end{proof}
\begin{remark}
In the limit $\delta\downarrow 0$, the conditions \eqref{StocOptimalityConditions} reduce to the conditions \eqref{detOptimalityConditions}.
\end{remark}
\begin{remark}
The stochastic dynamic version of the OMT as considered in Theorem \ref{thm:FirstOrderOptimalitySBPEE}, is known in the literature as the generalized Schr\"{o}dinger bridge problem. This class of problems originated in the works of Erwin Scr\"{o}dinger \cite{schrodinger1931umkehrung,schrodinger1932theorie,wakolbinger1990schrodinger} and as such predates both the mathematical theory of stochastic processes and feedback control. The qualifier ``generalized" refers to the presence of prior (in our case, bilinear) drift which was not considered in Schr\"{o}dinger's original investigations \cite{schrodinger1931umkehrung,schrodinger1932theorie}. In recent years, Schr\"{o}dinger bridge problems and their connections to OMT have come to prominence in both control \cite{chen2015optimal,chen2016relation,bakshi2020schrodinger,caluya2021wasserstein,chen2021stochastic,caluya2021reflected,nodozi2022schrodinger} and machine learning \cite{de2021diffusion,wang2021deep,chen2021likelihood} communities.
\end{remark}
While \eqref{StocOptimalityConditions} is valid for arbitrary (not necessarily small) $\delta>0$, we are particularly interested in numerically solving \eqref{StocOptimalityConditions} for small $\delta$ since then, its solution is guaranteed \cite{mikami2004monge,leonard2012schrodinger} to approximate the solution of \eqref{detOptimalityConditions}. Indeed, the second order terms in \eqref{StocOptimalityConditions} contribute toward smoother numerical solutions, i.e., behave as stochastic dynamic regularization in a computational sense. This idea of leveraging the stochastic version of the OMT for approximate numerical solution of the corresponding deterministic dynamic OMT has appeared, e.g., in \cite{haddad2020prediction}.   


\section{Solving the Conditions of Optimality\\using A Modified Physics Informed Neural Network}\label{secPINN}
We propose leveraging recent advances in neural network-based computational frameworks to numerically solve \eqref{StocOptimalityConditions} for small $\delta>0$. Specifically, we propose training a modified physics informed neural network (PINN) \cite{raissi2019physics,lu2021deepxde} to numerically solve \eqref{StocHJBPDE}-\eqref{StocBoundaryConditions}, which is a system of two second order coupled PDEs together with the endpoint PDF boundary conditions.

We point out here that one can alternatively use the Hopf-Cole \cite{hopf1950partial,cole1951quasi} a.k.a. Fleming's logarithmic transform \cite{fleming1982logarithmic} to rewrite the system \eqref{StocHJBPDE}-\eqref{StocBoundaryConditions} into a system of forward-backward Kolmogorov PDEs with the unknowns being the so-called ``Schr\"{o}dinger factors". Unlike \eqref{StocOptimalityConditions}, these PDEs are coupled via nonlinear boundary conditions; see e.g., \cite[Sec. II, III.B]{caluya2021wasserstein}, \cite[Sec. 5]{chen2021stochastic}. However, the numerical solution of the resulting system is then contingent on the availability of two initial value problem solvers: one for the forward Kolmogorov PDE and another for the backward Komogorov PDE. While specialized solvers may be designed for certain classes of prior nonlinear drifts \cite{caluya2019gradient,caluya2021wasserstein}, in general one resorts to particle-based methods such as Monte Carlo and Feynman-Kac solvers. Directly solving the conditions of optimality by adapting PINNs, as pursued here, offers an alternative computational method.

As in \cite{nodozi2022physics}, our training of PINN in this work involves minimizing a sum of four losses: two losses encoding the equation errors in \eqref{StocHJBPDE}-\eqref{Stocrho}, and the other two encoding the boundary condition errors in \eqref{StocBoundaryConditions}. However, different from \cite{nodozi2022physics}, we penalize the boundary condition losses using the discrete version of the Sinkhorn divergence \eqref{DefSinkDiv} computed using contractive Sinkhorn iterations \cite{cuturi2013sinkhorn}.

Because the Sinkhorn iterations involve a sequence of differentiable linear operations, it is Pytorch auto-differentiable to support backpropagation. Compared to the computationally demanding task of differentiating through a large linear program involving the Wasserstein losses, the Sinkhorn losses for the endpoint boundary conditions offer approximate solutions with far less computational cost allowing us to train the PINN on nontrivial problems.

\begin{figure*}[t]
\centering
\includegraphics[width=.8\paperwidth]{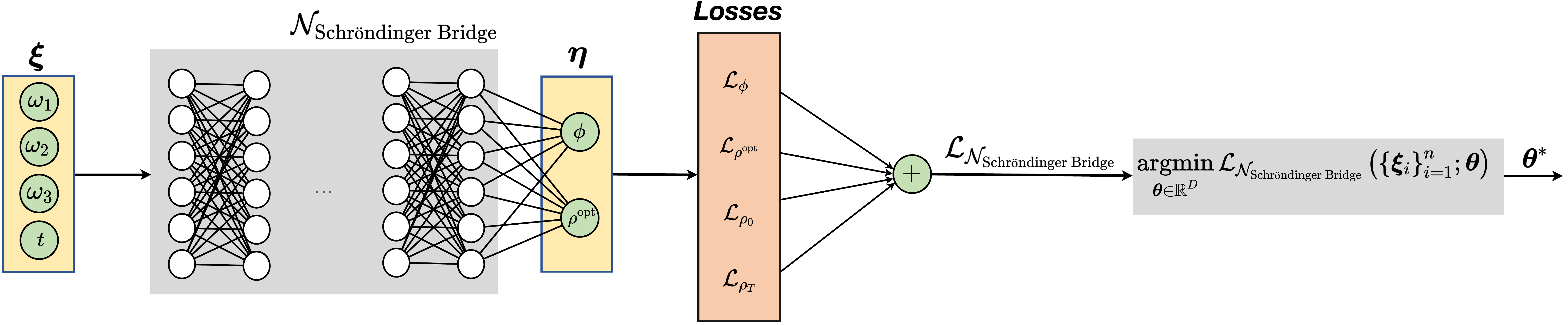}
\caption{{\small{The architecture of the PINN with $\bm{\xi}:=(\omega_1,\omega_2,\omega_3, t)$ as the input features. The PINN output $\bm{\eta}$ comprises of the value function and the  optimally controlled PDF, i.e., $\bm{\eta}:=(\phi,\rho^{\rm{opt}})$.}}}
\label{fig:PINNStructure}
\end{figure*}

 The proposed architecture of the PINN is shown in  Fig. \ref{fig:PINNStructure}. In our problem, $\bm{\xi}:=(\omega_1,\omega_2,\omega_3, t)$ comprises the features given to the PINN, and the PINN output $\bm{\eta}:=(\phi,\rho^{\rm{opt}})$. We parameterize the output of the fully connected feed-forward network via $\bm{\theta} \in \mathbb{R}^{D}$, i.e.,
\begin{align}
\bm{\eta}(\bm{\xi}) \approx \mathcal{N}_{\text{Schr\"{o}dinger Bridge}}(\bm{\xi} ; \bm{\theta}),
\end{align}
where $\mathcal{N}_{\text{Schr\"{o}dinger Bridge}}(\cdot;\bm{\theta})$ denotes the neural network approximant parameterized by $\bm{\theta}$, and $D$ is the dimension of the parameter space (i.e., the total number of to-be-trained weight, bias and scaling parameters for the network).

The overall loss function for the network denoted as $\mathcal{L}_{\mathcal{N}_{\text{Schr\"{o}dinger Bridge}}}$, consists of the sum of the equation error losses and the losses associated with the boundary conditions. Specifically, let $\mathcal{L}_{\phi}$ be the mean squared error (MSE) loss term for the HJB PDE \eqref{StocHJBPDE}, and let $\mathcal{L}_{\rho^{\rm{opt}}}$ be the MSE loss term for the FPK PDE \eqref{Stocrho}. For \eqref{StocBoundaryConditions}, we consider Sinkhorn regularized losses $\mathcal{L}_{\rho_{0}}$ and  $\mathcal{L}_{\rho_{T}}$. Then,
\begin{align}
\mathcal{L}_{\mathcal{N}_{\text{Schr\"{o}dinger Bridge}}}:=&\mathcal{L}_{\phi}+\mathcal{L}_{\rho^{\rm{opt}}}+\mathcal{L}_{\rho_{0}}+\mathcal{L}_{\rho_{T}},
\label{totalloss}
\end{align} 
where each summand loss term in \eqref{totalloss} is evaluated on a set of $n$ collocation points $\{\bm{\xi}_i\}_{i=1}^{n}$ in the domain of the feature space $\Omega:=\mathcal{X}\times[0,T]$ for some $\mathcal{X}\subset\mathbb{R}^{3}$, i.e., $\{\bm{\xi}_i\}_{i=1}^{n}\subset\Omega$.

We train the PINN with a Pytorch backend to compute the optimal training parameter
\begin{align}
\bm{\theta}^{*}:=\underset{\bm{\theta} \in \mathbb{R}^{D}}{\operatorname{argmin}}\: \mathcal{L}_{\mathcal{N}_{\text{Schr\"{o}dinger Bridge}}}(\{\bm{\xi}_i\}_{i=1}^{n}; \bm{\theta}).    
\label{NNtraining}    
\end{align}
In the next Section, we detail the simulation setup and report the numerical results.


\section{Numerical Simulations}\label{sec:numerical}
We consider the stochastic dynamics \eqref{SDE1} with $\delta=0.1$. The vector field $\bm{f}:\mathbb{R}^3 \mapsto \mathbb{R}^3$ is given in \eqref{DriftVectorField}. For the parameter vectors in \eqref{parameterVector}, we consider
$J_1=0.45$, $J_2=0.50$, and  $J_3=0.55$.

The control objective is to steer the prescribed joint PDF of the initial condition $\bm{x}(t=0) \sim \rho_0=\mathcal{N}\left(\bm{m}_0, \bm{\Sigma}_0\right)$ to the prescribed joint PDF of the terminal condition $\bm{x}(t=T) \sim \rho_T= \mathcal{N}\left(\bm{m}_T, \bm{\Sigma}_T\right)$ over $t \in[0,T]$, subject to \eqref{SDE1}, while minimizing \eqref{detOMTEErhoobj} with $q(\cdot)\equiv 0$, $r=\frac{1}{2}\|\cdot\|_2^2$. Here, we fix the final time $T = 4$ s, and 
$$
\begin{aligned}
\bm{m}_0=(2,2,2)^{\top},~ \bm{m}_T=(0,0,0)^{\top},~\bm{\Sigma}_0=\bm{\Sigma}_{T}=0.5\bm{I}_{3}.
\end{aligned}
$$
Due to the prior nonlinear drift, the optimally controlled transient joint state PDFs are expected to be non-Gaussian even when the endpoint joint state PDFs are Gaussian.

For training the $\mathcal{N}_{\text{Schr\"{o}dinger Bridge}}$, we use a network with 3 hidden layers with 70 neurons in each layer. The activation functions are chosen to be $\tanh(\cdot)$. The input-output structure of the network is as explained in Sec. \ref{secPINN}. 

We fix the state-time collocation domain $\Omega=\mathcal{X}\times [0,T] = [-5,5]^{3}\times [0,4]$. We trained the PINN for 80,000 epochs with the Adam optimizer \cite{kingma2014adam} and with a learning rate $10^{-3}$. We used $n=100,000$ pseudorandom samples (using Hammersley distribution) between the endpoint boundary conditions at $t=0$ and $t=T$ for the training. Additionally, to satisfy compute constraints, we uniformly randomly sampled 35,000 samples every 40,000 epochs. For computing the Sinkhorn losses at the endpoint boundary conditions, we use the entropic regularization parameter (see \eqref{DefSinkDiv}) $\varepsilon=0.1$.

Fig. \ref{fig:OptimallyControlledStateSamplePaths} depicts fifty optimally controlled state sample paths for this simulation. These sample paths are obtained via closed-loop simulation with the optimal control policy $\bm{u}^{{\rm{opt}}}$ resulting from the training of the PINN. 

Fig. \ref{fig.Marginal PDFs} shows the snaphsots of the univariate marginal PDFs under optimal control and the same without control, for the aforesaid numerical simulation. Following Sec. \ref{sec:uncontrolledPDF}, computing the uncontrolled PDFs for the deterministic dynamics (i.e., $\delta=0$) requires inverting \eqref{FlowMapAsymmetric}. We used the method-of-characteristics \cite{halder2011dispersion} to solve the corresponding unforced Liouville PDE, thereby obtaining the uncontrolled joint PDF snapshots. The marginals $\rho_i^{{\rm{unc}}}$, $i\in\llbracket 3\rrbracket$, in Fig. \ref{fig.Marginal PDFs} were obtained by numerically integrating these uncontrolled joints.



\begin{figure}[t]
\centering
\includegraphics[width=\linewidth]{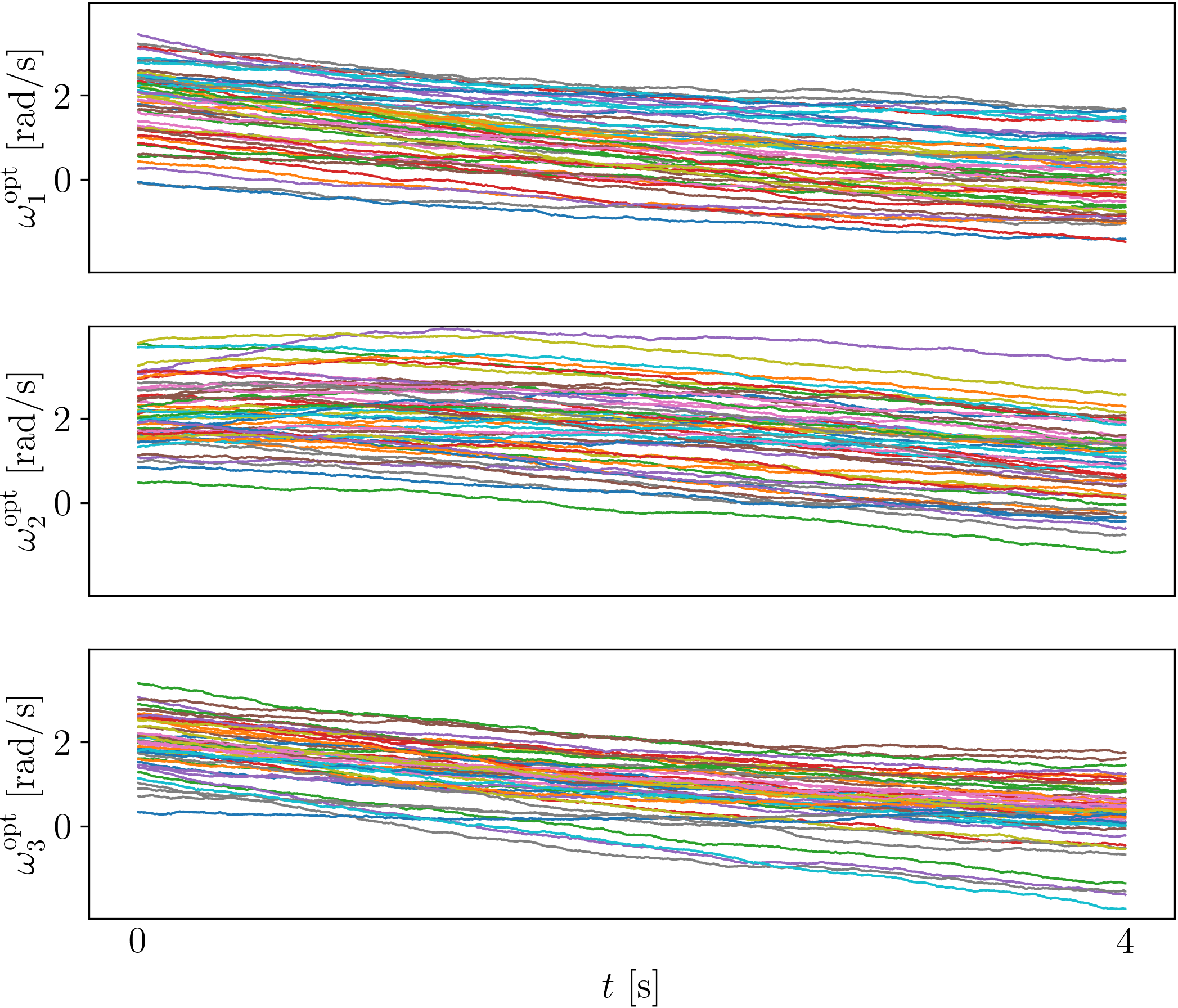}
\caption{{\small{Fifty optimally controlled closed-loop state sample paths $\omega_{i}^{{\rm{opt}}}(t)$, $i\in\llbracket 3\rrbracket$, for the simulation reported in Sec. \ref{sec:numerical}.}}}
\label{fig:OptimallyControlledStateSamplePaths}
\end{figure}
\begin{figure*}[htbp!]
    \centering
    \subfloat{{\includegraphics[width=5.45cm]{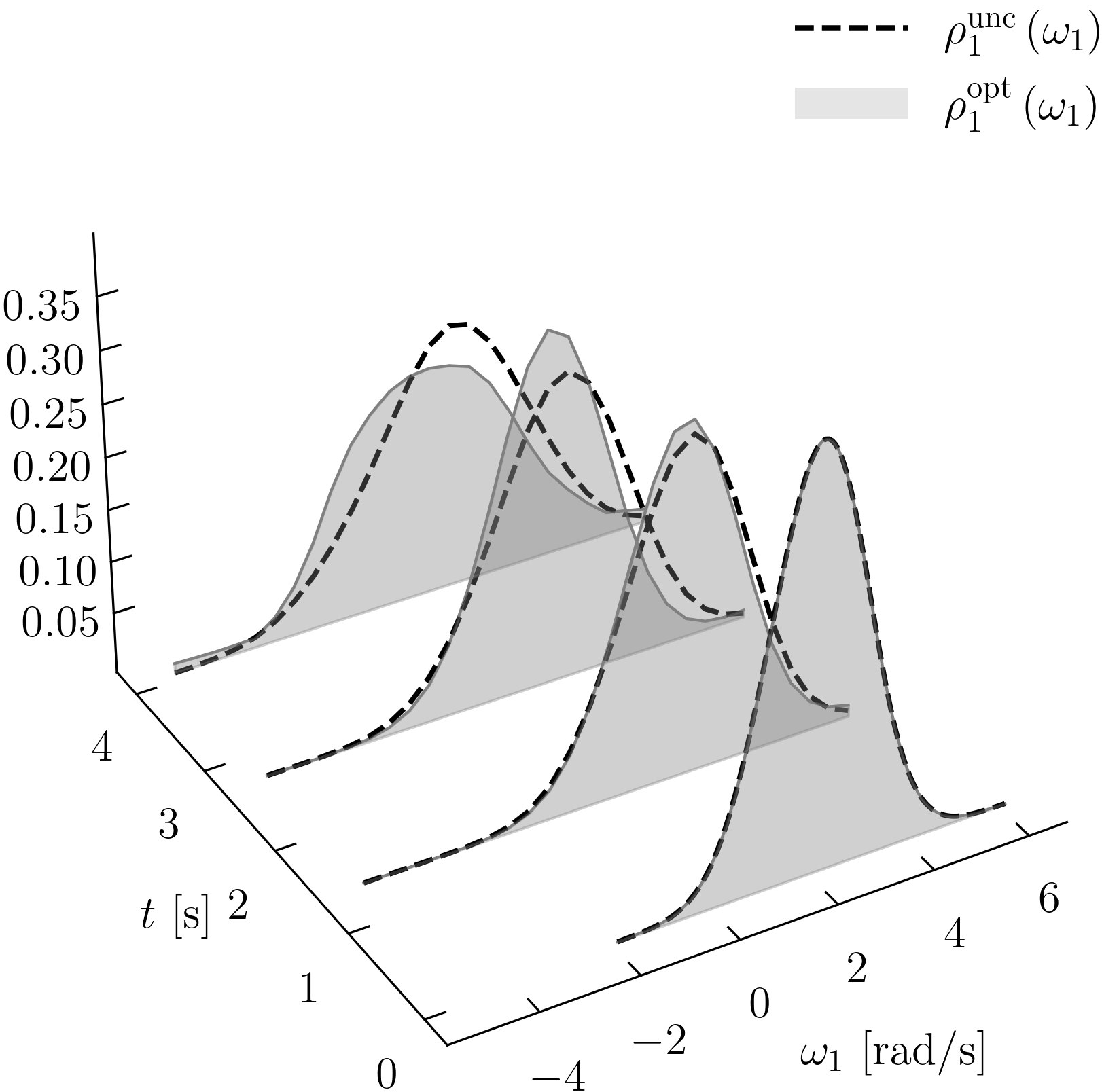}}}%
    ~
    \subfloat{{\includegraphics[width=5.45cm]{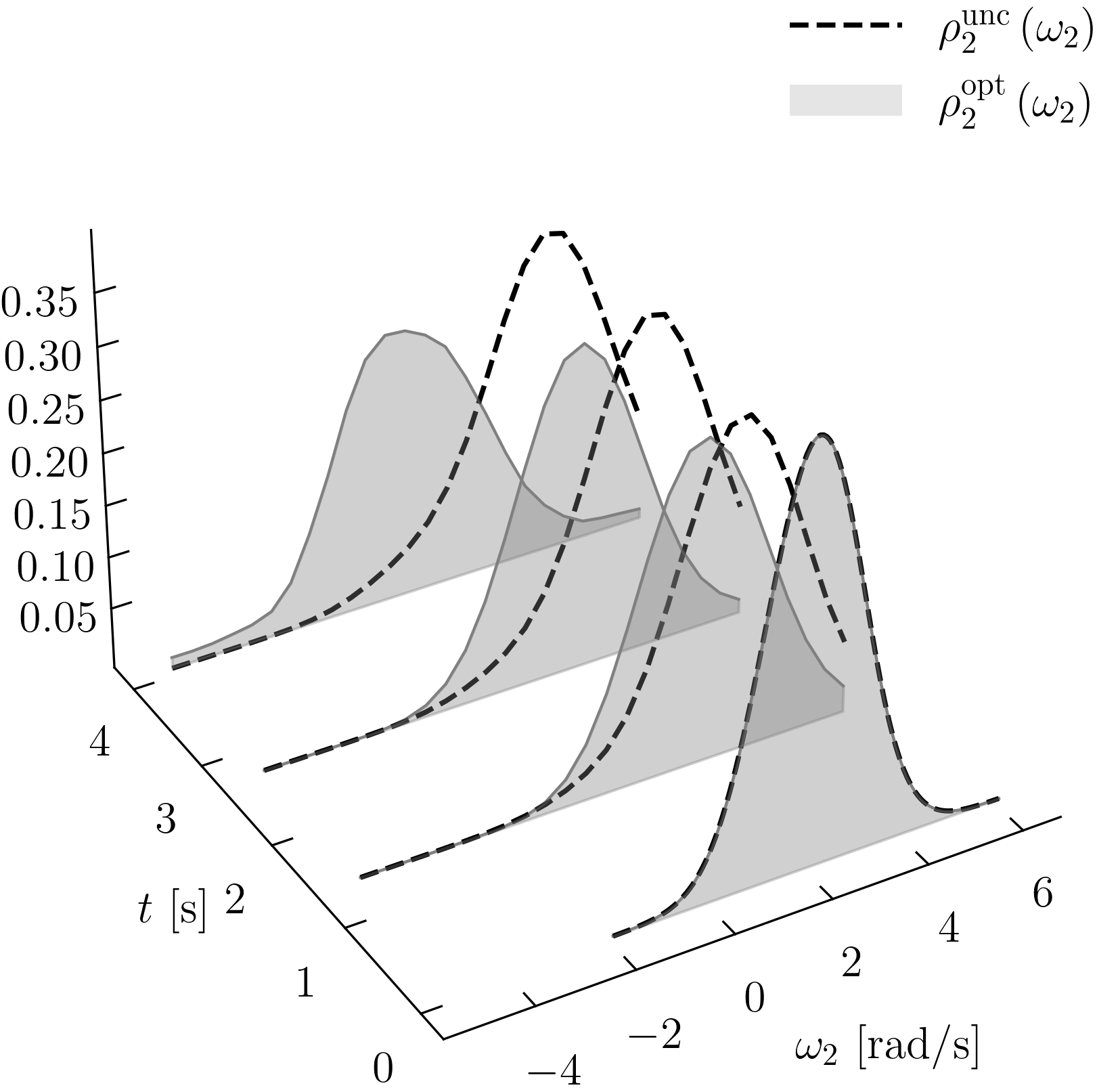}}}%
    ~
    \subfloat{{\includegraphics[width=5.45cm]{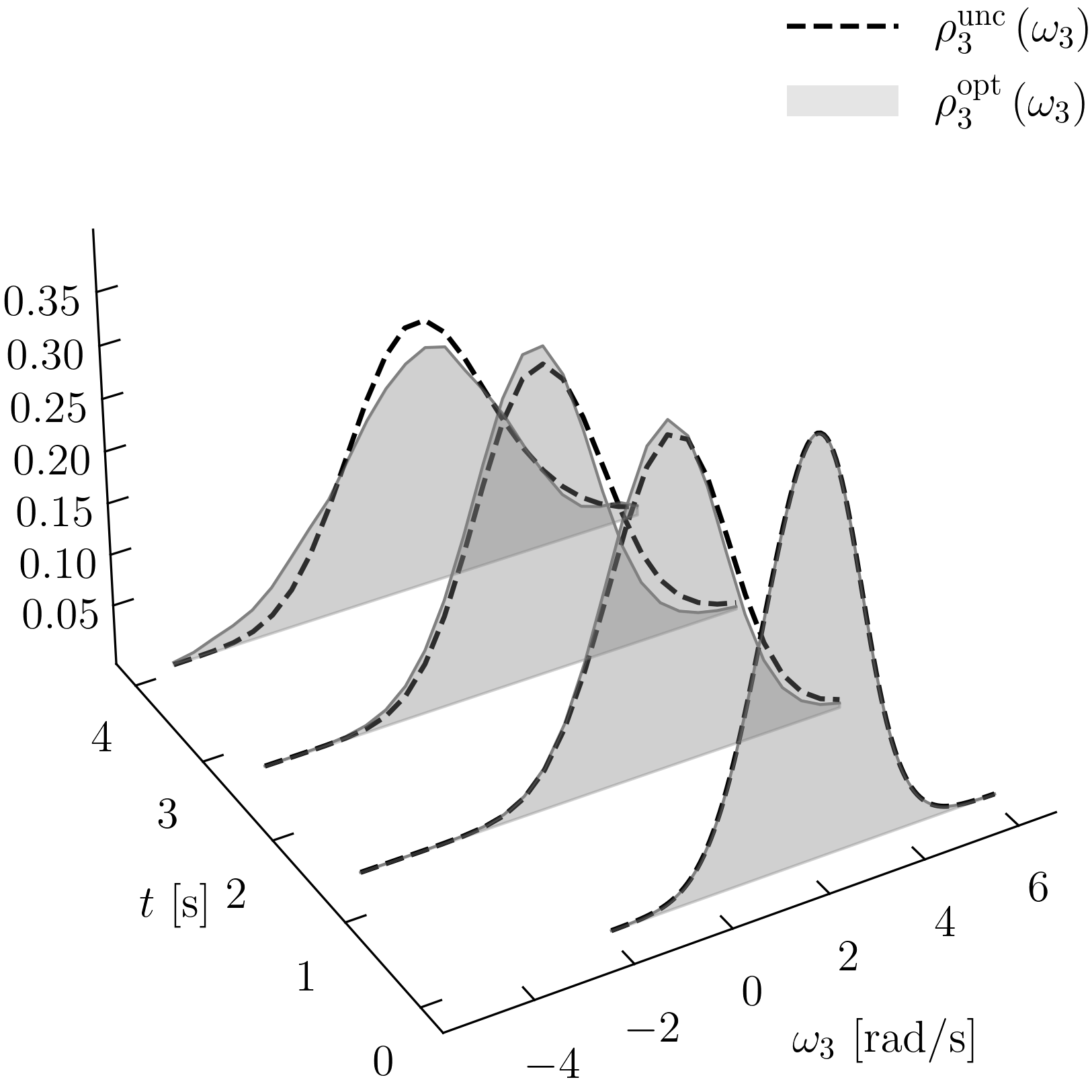}}}%
\caption{%
{\small{Four snapshots for the \emph{optimally controlled} univariate marginals $\rho_{i}^{{\rm{opt}}}$ and the corresponding \emph{uncontrolled} univariate marginals $\rho_{i}^{{\rm{unc}}}$, $i\in\llbracket 3\rrbracket$, for the numerical simulation in Sec. \ref{sec:numerical}.}}}%
\label{fig.Marginal PDFs}
\vspace*{-0.24in}
\end{figure*}







\section{Conclusions}\label{sec:conclusions}
We considered the optimal mass transport problem over the Euler equation governing the angular velocity dynamics. We studied both the deterministic and stochastic dynamic variants of this problem and explained their connections with the theory of classical optimal mass transport. We detailed the existence-uniqueness of solution as well as the necessary conditions of optimality. We provided an illustrative numerical example to demonstrate the solution of the optimal control synthesis using a modified physics informed neural network. The modification we propose involves differentiating through the Sinkhorn losses minimizing the boundary condition errors in the endpoint joint PDFs.

\bibliographystyle{IEEEtran}
\bibliography{References.bib}

\end{document}